\documentclass[reqno,12pt]{amsart}

\usepackage{amsmath, epsfig, cite}
\usepackage{amssymb}
\usepackage{amsfonts}
\usepackage{latexsym}
\usepackage{amsthm}

\newtheorem{theorem}{Theorem}[section]

\newtheorem{corollary}[theorem]{Corollary}
\newtheorem{lemma}[theorem]{Lemma}

\theoremstyle{remark}

\numberwithin{equation}{section}

\newcommand{\Z}{\mathbb Z}

\allowdisplaybreaks

\author{Haihong He}
\address{DEPARTMENT OF MATHEMATICS, SHANGHAI UNIVERSITY, SHANGHAI 200444, P. R. CHINA}
\email{$^*$ Corresponding author. xiaoxiawang@shu.edu.cn (X. Wang), hehaihong5@163.com(H. He).}
\author{Xiaoxia Wang$^*$}

\title[Two curious $q$-supercongruences and their extensions]{Two curious $q$-supercongruences and their extensions}


\thanks{This work is supported by Natural Science Foundation of Shanghai (22ZR1424100).}

\begin{document}

\begin{abstract}
We prove two single-parameter $q$-supercongruences   which were  recently  conjectured by Guo, and  establish their further extensions with one more parameter.   Crucial ingredients  in the proof are the terminating form of  $q$-binomial theorem    and  a  Karlsson-Minton type summation formula due to  Gasper. Incidentally,  an assertion of Wang, Li and Tang is also verified by establishing its $q$-analogue.
\end{abstract}

\maketitle

\section{Introduction}
 Except for investigating hypergeometric families of Calabi-Yau manifolds,  Rodriguez-Villegas \cite{RodriguezVillegas} also  observed (numerically) many possible supercongruences, including the following  one: for any  odd prime $p$,
\begin{align}\label{eq:Rodriguez-Villegas}
\sum_{k=0}^{p-1}\frac{(\frac{1}{2})_k^2}{k!^2}
\equiv (-1)^{(p-1)/2}
\pmod {p^2}.
\end{align}
Here and in what follows,  the the Pochhammer symbol is defined as and $(x)_0=1$ and  $(x)_n=x(x+1)\cdots(x+n-1)$  with $n$ a  positive integer.
The congruence \eqref{eq:Rodriguez-Villegas} was first proved by Mortenson \cite{Mortenson}.    Later,  Deines et al. \cite{Deines} gave the following generalization of \eqref{eq:Rodriguez-Villegas}: for any integer $d>1$ and prime $p\equiv1 \pmod d$,
\begin{align}
\sum_{k=0}^{p-1}\frac{(\frac{d-1}{d})_k^d}{k!^d}
\equiv-\Gamma_p\big(\tfrac 1d\big)^d
\pmod{p^2},
\label{generalization of RV}
\end{align}
where $\Gamma_p(x)$ denotes the $p$-adic Gamma function (cf. \cite[\S 1.12]{specialfunction}).

For any complex variable $x$ and   integer $n$, define the $q$\emph{-shifted factorial}   as
$$
(x;q)_\infty=\prod_{k=0}^{\infty}(1-xq^k) \quad \text{and}\quad
(x;q)_n=\frac{(x;q)_\infty}{(xq^n;q)_\infty}.
$$
For simplicity,  we also adopt the compact expression $$
(x_1,x_2,\ldots,x_m;q)_n=(x_1;q)_n (x_2;q)_n\cdots (x_m;q)_n.$$
Moreover, $[n]:=[n]_q=1+q+\cdots+q^{n-1}$ denotes the $q$\emph{-integer}.
A $q$-analogue of \eqref{eq:Rodriguez-Villegas} can be found  in Guo and Zeng \cite[Corollary 2.4]{GuoZhen2014}.
Recently,  Guo \cite[Theorem 1.1]{Guoqdrrama2022} established a $q$-extension of \eqref{generalization of RV}: for any integers $d, n>1$ with $n\equiv1 \pmod d$,
\begin{align} \label{aa}
\sum_{k=0}^{n-1}\frac{(q^{d-1};q^d)_k^dq^{dk}}{(q^d;q^d)_k^d}
\equiv \frac{(q^d;q^d)_{(d-1)(n-1)/d}q^{(d-1)(n-1)(d+n-1)/(2d)}}{(q^d;q^d)_{(n-1)/d}^{d-1}(-1)^{(d-1)(n-1)/d}}\pmod{\Phi_n(q)^2}.
\end{align}
Here and in what follows, the $n$-th \emph{cyclotomic polynomial} $\Phi_n(q)$ is given by
\begin{align*}
\Phi_n(q)=\prod_{\substack{1\leq k \leq n \\ \gcd(k,n)=1}}(q-\zeta^k),
\end{align*}
where $\zeta$  is an $n$-th primitive root of unity.
In the same paper, Guo also presented two analogous $q$-supercongruences as follows.\\
$(\romannumeral1)$ For any odd integer $d\geq3$ and  integer $n$ with $n\equiv-1 \pmod d$ and $n\geq 2d-1$,
\begin{align}\label{odd}
&\sum_{k=0}^{n-1}\frac{(q^{d+1};q^d)_k^{d-1}(q^{1-d};q^d)_kq^{dk}}{(q^d;q^d)_k^d}\notag\\
&\equiv \frac{(1-q)(1-q^{d-1})(q^d;q^d)_{n-1-(n+1)/d}}{-(q^d;q^d)_{(n+1)/d}^{d-1}}q^{(d(d+n)(n+1)-(n+1)^2)/(2d)-1}\pmod{\Phi_n(q)^2}.
\end{align}
$(\romannumeral2)$ For any even integer $d\geq4$ and positive integer $n$ with $n\equiv-1 \pmod d$,
\begin{align}
&\sum_{k=0}^{n-1}\frac{(q^{d+1};q^d)_k^{d-2}(q;q^d)_k^2q^{dk}}{(q^d;q^d)_k^d}\notag\\
&\equiv \frac{(1-q)^2(q^d;q^d)_{n-1-(n+1)/d}}{(-1)^{(n+1)/d}(q^d;q^d)_{(n+1)/d}^{d-1}}q^{(d(d+n)(n+1)-(n+1)^2)/(2d)-2}\pmod{\Phi_n(q)^2}.
\label{similar to}
\end{align}
For  recent progress on congruences and $q$-congruences, we recommend the  literatures \cite{XuWang, GuoA2, GuoSchlosser, XuWang1, liuwang22, WangPan, WangYu, Wei, gn, glm, Weica} to  readers.

As the complements of \eqref{odd} and \eqref{similar to}, Guo \cite[Conjectures 5.2 and 5.3]{Guoqdrrama2022} proposed   two $q$-supercongruences lined as the following two theorems respectively.   Moreover,  we shall establish the common generalization of \eqref{odd} and \eqref{eq:conjecture 5.2} and also that of \eqref{similar to} and  \eqref{eq:conjecture 5.3}  with two   parameters.


\begin{theorem}[6, Conjecture 5.2]\label{thm:2}
Let $d\geq4$ be an even integer and $n$ an integer with $n\equiv-1\pmod d$ and $n\geq2d-1$. Then,
\begin{align}
&\sum_{k=0}^{n-1}\frac{(q^{d+1};q^d)_k^{d-1}(q^{1-d};q^d)_kq^{dk}}{(q^d;q^d)_k^d}\notag\\
&\equiv \frac{(1-q)(1-q^{d-1})(q^d;q^d)_{n-1-(n+1)/d}}{-(-1)^{(n+1)/d}(q^d;q^d)_{(n+1)/d}^{d-1}}q^{(d(d+n)(n+1)-(n+1)^2)/(2d)-1}\pmod{\Phi_n(q)^2}.
\label{eq:conjecture 5.2}
\end{align}
\end{theorem}
\begin{theorem}[6, Conjecture 5.3]\label{theorem:1}
Let $d\geq3$ be an odd integer and $n$ a positive integer with $n\equiv-1\pmod d$. Then,
\begin{align}
&\sum_{k=0}^{n-1}\frac{(q^{d+1};q^d)_k^{d-2}(q;q^d)_k^2q^{dk}}{(q^d;q^d)_k^d}\notag\\
&\equiv \frac{(1-q)^2(q^d;q^d)_{n-1-(n+1)/d}}{(q^d;q^d)_{(n+1)/d}^{d-1}}q^{(d(d+n)(n+1)-(n+1)^2)/(2d)-2}\pmod{\Phi_n(q)^2}.
\label{eq:conjecture 5.3}
\end{align}
\end{theorem}

Particularly,  letting $n=p$ be an odd prime and $q\rightarrow1$ in Theorems \ref{thm:2} and \ref{theorem:1}, we arrive at Guo's two congruences \cite[(5.4) and (5.5)]{{Guoqdrrama2022}}, which have been successfully proved by Wang, Li and Tang \cite{WangLiTang}. Besides,  the authors \cite{WangLiTang}  conjectured that:
for  any integers $d\geq2$, $n$  with $n\equiv -1 \pmod d$ and $n\geq 2d-1$,
\begin{align}\label{equality-Wang-Li-Tang}
\frac{{(n-1)!}^dd^{dn-d}}{n^2}
\sum_{k=0}^{n-1}\frac{(\frac{d+1}{d})_k^{d-2}(\frac{1}{d})_k(\frac{1-d}{d})_k}{(1)_k^d}\in \Z.
\end{align}
The  assertion \eqref{equality-Wang-Li-Tang} shall be proved  by determining the following $q$-congruence.
\begin{theorem}\label{Theorem-Wang-Li-Tang}
Let $d\geq2$ and $n$ be integers with $n\equiv -1 \pmod d$ and $n\geq 2d-1$. Then,
\begin{align*}
\frac{(q^d;q^d)_{n-1}^d}{(1-q)^{dn-d}} \sum_{k=0}^{n-1}\frac{(q^{d+1};q^d)_k^{d-2}(q,q^{1-d};q^d)_kq^{dk}}{(q^d;q^d)_k^d}
\equiv 0\pmod{[n]^2}.
\end{align*}
\end{theorem}

Clearly, \eqref{equality-Wang-Li-Tang} follows by letting $q \rightarrow 1$ in Theorem \ref{Theorem-Wang-Li-Tang}.

The rest of this paper is  organized as follows. We shall employ the terminating form of  $q$-binomial theorem  to  confirm Theorems \ref{thm:2} and \ref{Theorem-Wang-Li-Tang} in Section 2. Section 3 is devoted to proving Theorem \ref{theorem:1} by means of Gasper's Karlsson-Minton type summation formula.   In Section 4, we shall present two-parametric  extensions of \eqref{odd} and \eqref{eq:conjecture 5.2}, together with \eqref{odd}  and  \eqref{eq:conjecture 5.3}. The `creative microscoping' method introduced by Guo and Zudilin \cite{qmicroscope2019} will also be  employed in the proofs.

\section{Proofs of Theorems \ref{thm:2} and \ref{Theorem-Wang-Li-Tang}}
To prove Theorems \ref{thm:2} and \ref{Theorem-Wang-Li-Tang},  we require the following key result.
\begin{lemma}\label{lemma:1}
Let $d, r$ be positive integers with $d\geq3+r$ and $\gcd(d,r)=1$. Let $n$ be an integer such that $n\equiv-r\pmod d$ and $n\geq 2d-r$. Then,
\begin{align}
\sum_{k=0}^{n-1}\frac{(q^{d+r};q^d)_k^{d-r-1}(q^r;q^d)_k^{r}(q^{r-d};q^d)_kq^{dk}}{(q^d;q^d)_k^d}
\equiv 0\pmod{\Phi_n(q)^2}.
\label{eq:lemma1}
\end{align}
Especially,  for integers $d\geq 4$ and $n\geq 2d-1$ with $n\equiv-1\pmod d$, we have,
\begin{align}\label{r=1}
\sum_{k=0}^{n-1}\frac{(q^{d+1};q^d)_k^{d-2}(q,q^{1-d};q^d)_kq^{dk}}{(q^d;q^d)_k^d}
\equiv 0\pmod{\Phi_n(q)^2}.
\end{align}
\end{lemma}
\begin{proof}
Obviously,  the ratio  $(q^{d+r};q^d)_k^{d-r-1}/(q^d;q^d)_k^d$ contains the factor $(1-q^{(d-1)n})^2$ for $n-(n+r)/d\leq k\leq n-1$.
Besides, it is easy to find that    for $k\geq0$,
\begin{align*}
(q^r;q^d)_k^r(q^{r-d};q^d)_k=(1-q^{r-d})(1-q^r)^{r+1}(q^{d+r};q^d)_{k-2}^{r+1}(1-q^{dk-d+r})^r.
\end{align*}
Equivalently, we just prove the truth of the subsequent congruence,
\begin{align}\label{eq:equality of lemma1}
\sum_{k=0}^{n-1-(n+r)/d}\frac{(q^{d+r};q^d)_k^{d-r-1}(q^{d+r};q^d)_{k-2}^{r+1}(1-q^{dk-d+r})^rq^{dk}}{(q^d;q^d)_k^d}\equiv0\pmod{\Phi_n(q)^2}.
\end{align}
Before that, we  confirm
the following parametric generalizations of \eqref{eq:equality of lemma1}, which are both symmetric about $a$ and $a^{-1}$:
\begin{align}
&\sum_{k=0}^{n-1-(n+r)/d}\frac{(a^{d-1}q^{d+r},a^{d-3}q^{d+r},\cdots,a^{r+2}q^{d+r},a^{-r-2}q^{d+r},\cdots,a^{3-d}q^{d+r},a^{1-d}q^{d+r};q^d)_kq^{dk}}
{(q^d,a^{d-2}q^d,a^{d-4}q^d,\cdots,a^{r+3}q^d,a^{-r-1}q^d,\cdots,a^{4-d}q^d,a^{2-d}q^d;q^d)_k}\notag\\\label{eq:parametric form of equality of lemma1}
&\qquad\quad\times\frac{(a^rq^{d+r},a^{r-2}q^{d+r},\cdots,a^{-r+2}q^{d+r},a^{-r}q^{d+r};q^d)_{k-2}(1-q^{dk-d+r})^{r}}
{(a^{r+1}q^d,a^{r-1}q^d,\cdots,a^{-r+3}q^d,a^{-r+1}q^d;q^d)_k}\notag\\
&\qquad\equiv0
\pmod{(1-aq^n)(a-q^n)}
\end{align}
for $d+r\equiv 1\pmod{2}$, and
\begin{align}
&\sum_{k=0}^{n-1-(n+r)/d}\frac{(a^{d-1}q^{d+r},a^{d-3}q^{d+r},\cdots,a^{r+3}q^{d+r},q^{d+r},a^{-r-3}q^{d+r},\cdots,a^{3-d}q^{d+r},a^{1-d}q^{d+r};q^d)_k}
{(q^d,a^{d-2}q^d,a^{d-4}q^d,\cdots,a^{r+4}q^d,aq^d,a^{-r-2}q^d,\cdots,a^{4-d}q^d,a^{2-d}q^d;q^d)_kq^{-dk}}\notag\\   \label{d odd}
&\qquad\quad\times\frac{(a^{r+1}q^{d+r},a^{r-1}q^{d+r},\cdots,a^2q^{d+r},a^{-2}q^{d+r},\cdots,a^{-r+1}q^{d+r},a^{-r-1}q^{d+r};q^d)_{k-2}}
{(a^{r+2}q^d,a^rq^d,\cdots,a^3q^d,a^{-1}q^d,\cdots,a^{-r+2}q^d,a^{-r}q^d;q^d)_k(1-q^{dk-d+r})^{-r}}\notag\\
&\qquad\equiv0
\pmod{(1-aq^n)(a-q^n)}
\end{align}
for $d\equiv r\equiv 1\pmod{2}$.

Next, we concentrate on the proof of \eqref{eq:parametric form of equality of lemma1}. Since  $\gcd(d,r)=\gcd(d,n)=1$,
the denominator of the left-hand side of \eqref{eq:parametric form of equality of lemma1} does not contain the factor $1-aq^n$ or $a-q^n$.
Thus, for $a=q^{-n}$ or $a=q^n$, noticing that $(q^{d+r-(d-1)n};q^d)_k=0$ for $k>n-1-(n+r)/d$,    the left-hand side of \eqref{eq:parametric form of equality of lemma1}  equals
\begin{align}\label{polynomial}
&\sum_{k=0}^{n-1-(n+r)/d}\frac{(q^{d+r-(d-1)n},q^{d+r-(d-3)n},\cdots,q^{d+r-(r+2)n},q^{d+r+(r+2)n},\cdots,q^{d+r+(d-1)n};q^d)_kq^{dk}}
{(q^d,q^{d-(d-2)n},\cdots,q^{d-(r+3)n},q^{d+(r+1)n},\cdots,q^{d+(d-2)n};q^d)_k}\notag\\
&\qquad\quad\times\frac{(q^{d+r-rn},q^{d+r-(r-2)n},\cdots,q^{d+r+(r-2)n},q^{d+r+rn};q^d)_{k-2}(1-q^{dk-d+r})^{r}}
{(q^{d-(r+1)n},q^{d-(r-1)n},\cdots,q^{d+(r-3)n},q^{d+(r-1)n};q^d)_k}\notag\\
&\quad\quad=\sum_{k=0}^{n-1-(n+r)/d}(-1)^kq^{d\binom{n-1-(n+r)/d-k}{2}}{n-1-(n+r)/d\brack k}_{q^d}P(q^{dk}).
\end{align}
Here $P(q^{dk})$ is a polynomial in $q^{dk}$ of degree $n-2-(n+r)/d$ and
$$
{n\brack k}:={n\brack k}_q=\frac{(q;q)_n}{(q;q)_k(q;q)_{n-k}}
$$
denotes the $q$\emph{-binomial coefficient}.
In the derivation of \eqref{polynomial}, we have utilized the following formulas, which can be certified readily:
\begin{align*}
\frac{(q^{d+r-(d-1)n};q^d)_kq^{dk}}{(q^d;q^d)_k}&=(-1)^k{n-1-(n+r)/d\brack k}_{q^d}q^{d\binom{k}{2}+(n+2d+r-dn)k},
\\
d\binom{k}{2}+(n+2d+r-dn)k&=d\binom{n-1-(n+r)/d-k}{2}-d\binom{n-1-(n+r)/d}{2},
\end{align*}
for $1\leq{j}\leq {d-1}$ and $j\neq (d-r-1)/2,\cdots,(d+r-1)/2$
\begin{align}
\frac{(q^{d+r-(d-2j-1)n};q^d)_k}{(q^{d-(d-2j)n};q^d)_k}&=\frac{(q^{d-(d-2j)n+dk};q^d)_{(n+r)/d}}{(q^{d-(d-2j)n};q^d)_{(n+r)/d}},
\label{eq:3}
\end{align}
and for $(d-r-1)/2\leq{j}\leq (d+r-1)/2$
\begin{align}
\frac{(q^{d+r-(d-2j-1)n};q^d)_{k-2}}{(q^{d-(d-2j)n};q^d)_k}&=\frac{(q^{d-(d-2j)n+dk};q^d)_{(n+r)/d-2}}{(q^{d-(d-2j)n};q^d)_{(n+r)/d}}.
\label{eq:2}
\end{align}
We see that the right-hand sides of \eqref{eq:3} and \eqref{eq:2}  are polynomials in $q^{dk}$ of degree  $(n+r)/d$ and $(n+r)/d-2$ respectively.
Then, by the terminating form of $q$-binomial theorem (see, for example,  \cite[p. 36]{Andrews-qbinomial}):
\begin{align*}
\sum_{k=0}^{n}(-1)^k{n\brack k}q^{\binom{n-k}{2}+jk}=0 \quad\text{ for } 0\leq j \leq n-1,
\end{align*}
we deduce that the right-hand side of \eqref{polynomial} equals zero.
This implies that the congruence \eqref{eq:parametric form of equality of lemma1} holds modulo $(1-aq^n)(a-q^n)$.
Similarly, we are able to  prove    \eqref{d odd} and we omit the details here.

Apparently,   the limit  $(1-aq^n)(a-q^n)$  has the factor $\Phi_n(q)^2$ as $a\rightarrow 1$. Moreover, the denominators on the left-hand side of \eqref{eq:parametric form of equality of lemma1} and \eqref{d odd} are both coprime with  $\Phi_n(q)$ as $a\rightarrow 1$.   Hence, letting  $a\rightarrow 1$ in \eqref{eq:parametric form of equality of lemma1} and \eqref{d odd}, we are led to \eqref{eq:equality of lemma1}.  Thus,   we verify the correction of \eqref{eq:lemma1}.
\end{proof}

Now, we are ready to prove Theorems \ref{thm:2} and \ref{Theorem-Wang-Li-Tang}.
\begin{proof}[Proof of Theorem \ref{thm:2}]
It is routine to check that: for any  integer $d\geq 2$,
\begin{align}\label{23}
&\sum_{k=0}^{n-1}\frac{(q^{d+1};q^d)_k^{d-1}(q^{1-d};q^d)_kq^{dk}}{(q^d;q^d)_k^d}\notag\\
&=[d]\sum_{k=0}^{n-1}\frac{(q^{d+1};q^d)_k^{d-2}(q,q^{1-d};q^d)_kq^{dk}}{(q^d;q^d)_k^d}
-q[d-1]\sum_{k=0}^{n-1}\frac{(q^{d+1};q^d)_k^{d-2}(q;q^d)_k^2q^{dk}}{(q^d;q^d)_k^d}
,
\end{align}
where we have applied the following identity:
\begin{align*}
(q^{d+1},q^{1-d};q^d)_k=-q[d-1]\left(1+\frac{1-q^d}{q^d-q^{dk+1}}\right)(q;q^d)_k^2.
\end{align*}
Hence, from \eqref{similar to}, \eqref{r=1} and \eqref{23}, we arrive at Theorem \ref{thm:2}.
\end{proof}

\begin{proof}[Proof of Theorem \ref{Theorem-Wang-Li-Tang}]
At first, we claim that the following two $q$-congruences hold, which are the $d=2, 3$ cases of  \eqref{r=1}: for any integer  $n$,  modulo $\Phi_n(q)^2$,
\begin{align}\label{22}
\sum_{k=0}^{n-1}\frac{(q,q^{-1};q^2)_kq^{2k}}{(q^2;q^2)_k^2}
&\equiv 0\quad \text{if}~n\geq 3~\text{and}~n\equiv1\pmod 2,\\\label{33}
\sum_{k=0}^{n-1}\frac{(q^4,q,q^{-2},;q^3)_kq^{3k}}{(q^3;q^3)_k^3}
&\equiv 0\quad\text{if}~n\geq 5~\text{and}~n\equiv 2\pmod 3.
\end{align}
In fact, taking $d=2$ in \eqref{23}, we have, modulo $\Phi_n(q)^2$,
\begin{align}\label{231}
[2]\sum_{k=0}^{n-1}\frac{(q,q^{-1};q^2)_kq^{2k}}{(q^2;q^2)_k^2}
&=\sum_{k=0}^{n-1}\frac{(q^{3},q^{-1};q^2)_kq^{2k}}{(q^2;q^2)_k^2}+q\sum_{k=0}^{n-1}\frac{(q;q^2)_k^2q^{2k}}{(q^2;q^2)_k^2}\notag\\[1.5mm]
&\equiv (-1)^{(n+1)/2}q^{(n^2+3)/4}+q(-1)^{(n-1)/2}q^{(n^2-1)/4},
\end{align}
where   we have used a $q$-supercongruence from the work \cite[p. 9]{Guoqdrrama2022} and the $d=2$ case of \eqref{aa}.
Analogously, letting  $d=3$ in \eqref{23}, we have, modulo $\Phi_n(q)^2$,
\begin{align}\label{232}
[3]\sum_{k=0}^{n-1}\frac{(q^4,q,q^{-2};q^3)_kq^{3k}}{(q^3;q^3)_k^3}
&=\sum_{k=0}^{n-1}\frac{(q^{4},q^4,q^{-2};q^3)_kq^{3k}}{(q^3;q^3)_k^3}
+q[2]\sum_{k=0}^{n-1}\frac{(q^4,q,q;q^3)_kq^{3k}}{(q^3;q^3)_k^3}\notag\\
&\equiv \frac{(1-q)(1-q^2)(q^3;q^3)_{(2n-1)/3-1}}{-(q^3;q^3)_{(n+1)/3}^2}q^{(3(3+n)(n+1)-(n+1)^2)/6-1}\notag\\
&\quad+q[2]\frac{(1-q)^2(q^3;q^3)_{(2n-1)/3-1}}{(q^3;q^3)_{(n+1)/3}^2}q^{(3(3+n)(n+1)-(n+1)^2)/6-2},
\end{align}
where we have utilized   \eqref{odd} and \eqref{eq:conjecture 5.3} with  $d=3$.  Therefore, the $q$-results \eqref{22} and \eqref{33},  respectively,  follow   from \eqref{231} and \eqref{232}. In addition, we notice that: for any integer  $d\geq 2$,
\begin{align}\label{n}
\frac{(q^d;q^d)_{n-1}^d}{(1-q)^{dn-d}}=\prod_{1\leq m< n}\bigg(\frac{1-q^{md}}{1-q}\bigg)^d\equiv 0\pmod{\prod_{\substack{1< m<n \\ m\mid n}}\Phi_m(q)^2}.
\end{align}
Combining the following property
\begin{align*}
[n]=\Phi_n(q)\prod_{\substack{1< m<n \\ m\mid n}}\Phi_m(q)
\end{align*}
 with the  congruences   \eqref{r=1}, \eqref{22}, \eqref{33} and \eqref{n},
we  obtain Theorem \ref{Theorem-Wang-Li-Tang}.
\end{proof}

\section{Proof of Theorem \ref{theorem:1}}

\begin{proof}
We shall prove Theorem \ref{theorem:1} by using  Gasper's Karlsson-Minton type summation formula (see \cite[(1.9.9)]{GasperRahman2014} for   more general form): for nonnegative integers $N,n_1,\cdots,n_m$ with   $N=n_1+\cdots+n_m$,
\begin{align}
\sum_{k=0}^N\frac{(q^{-N},b_1q^{n_1},\cdots,b_mq^{n_m};q)_k}{(q,b_1,\cdots,b_m;q)_k}q^k
=(-1)^N\frac{(q;q)_Nb_1^{n_1}\cdots b_m^{n_m}}{(b_1;q)_{n_1}\cdots(b_m;q)_{n_m}}q^{\binom{n_1}{2}+\cdots+\binom{n_m}{2}}.
\label{eq:KM}
\end{align}
Below, the two conditions $d>3$ and $d=3$ of Theorem \ref{theorem:1} shall be discussed separately.\\
{$(\romannumeral1)$} For $d>3$,  we consider the following   generalization of \eqref{eq:conjecture 5.3}:  modulo $(1-aq^n)(a-q^n)$,
\begin{align}
&\sum_{k=0}^{n-1}\frac{(a^{d-1}q^{d+1},a^{d-3}q^{d+1},\cdots,a^4q^{d+1},a^2q,q^{d+1},a^{-2}q,a^{-4}q^{d+1},\cdots,a^{3-d}q^{d+1},a^{1-d}q^{d+1};q^d)_k}
{(a^{d-2}q^d,a^{d-4}q^d,\cdots,a^{4-d}q^d,a^{2-d}q^d;q^d)_k(q^d;q^d)_kq^{-dk}}\notag\\
&\equiv \frac{(1-a^2q)(1-q/a^{2})(q^d;q^d)_{n-1-(n+1)/d}}
{(a^{d-2}q^d,a^{d-4}q^d,\cdots,a^{4-d}q^d,a^{2-d}q^d;q^d)_{(n+1)/d}}q^{(d(d+n)(n+1)-(n+1)^2)/(2d)-2}.
\label{eq:d>3 parametric conjecture 5.3}
\end{align}
In fact,  we take the substitutions
$q\rightarrow{q^d}$, $m=d-1$, $ N=n-1-(n+1)/d$, $b_j=q^{d-(d-2j)n}$ $(1\leq j\leq d-1)$, $n_{(d-3)/2}=n_{(d+1)/2}=(n+r)/d-1$ and $n_j=(n+r)/d$ $(1\leq j\leq d-1$ and $j\neq (d-3)/2$, $(d+1)/2)$ in  \eqref{eq:KM}. Thus, for $a=q^{-n}$ or $a=q^n$, the left-hand side of \eqref{eq:d>3 parametric conjecture 5.3} can be simplified as
\begin{align*}
&\frac{(-1)^{n-1-(n+1)/d}q^{(d-1)(n+1)-2(d-n)+d(d-3)\binom{(n+1)/d}{2}+2d\binom{(n+1)/d-1}{2}}}
{(q^{d-(d-2)n},q^{d-(d-4)n},\cdots,q^{d-5n},q^{d-n},q^{d+3n},\cdots,q^{d+(d-4)n},q^{d+(d-2)n};q^d)_{(n+1)/d}}\notag\\
&\quad\times \frac{(q^d;q^d)_{n-1-(n+1)/d}}{(q^{d-3n},q^{d+n};q^d)_{(n+1)/d-1}}\notag\\
&=\frac{(1-q^{1-2n})(1-q^{1+2n})(q^d;q^d)_{n-1-(n+1)/d}q^{(d(d+n)(n+1)-(n+1)^2)/(2d)-2}}
{(q^{d-(d-2)n},q^{d-(d-4)n},\cdots,q^{d+(d-4)n},q^{d+(d-2)n};q^d)_{(n+1)/d}}.
\end{align*}
Namely, the $q$-congruence \eqref{eq:d>3 parametric conjecture 5.3} is true modulo $(1-aq^n)(a-q^n)$.
\\
{$(\romannumeral2)$} For $d=3$,  we regard the following $q$-congruence:   modulo $(1-aq^n)(a-q^n)$,
\begin{align}\label{eq:parametric d=3}
\sum_{k=0}^{n-1}\frac{(a^2q,q^4,q/a^2;q^3)_kq^{3k}}{(aq^3,q^3/a,q^3;q^3)_k}
\equiv\frac{(1-a^2q)(1-q/a^2)(q^3;q^3)_{n-1-(n+1)/3}}{(aq^3,q^3/a;q^3)_{(n+1)/3}}q^{(n^2+5n-2)/3}.
\end{align}
Actually, the $q$-congruence \eqref{eq:parametric d=3} follows by choosing
$q\rightarrow{q^3}$, $m=2$, $N=(2n-1)/3$, $b_1=q^{3-n}$, $b_2=q^{3+n}$, $n_1=(n+1)/3$ and $n_2=(n+1)/3-1$ in \eqref{eq:KM}.
Hence,   we finish the proof of  Theorem \ref{theorem:1}  by letting $a\rightarrow 1$ in \eqref{eq:d>3 parametric conjecture 5.3} and \eqref{eq:parametric d=3}.
\end{proof}

\section{The extensions of Theorems \ref{thm:2} and \ref{theorem:1}}
In this section, we present the common extension of \eqref{odd} and \eqref{eq:conjecture 5.2} and also that of  \eqref{similar to} and \eqref{eq:conjecture 5.3} with two parameters. Our discoveries are as follows.
\begin{theorem}\label{thm:4}
Let $d, r$ be positive integers with $d\geq3+r$ and $\gcd(d,r)=1$.  Let $n$ be an integer such that $n\equiv-r\pmod d$ and $n\geq 2d-r$. Then, modulo $\Phi_n(q)^2$,
\begin{align}\label{4.1}
\sum_{k=0}^{n-1}\frac{(q^{d+r};q^d)_k^{d-r}(q^r;q^d)_k^{r-1}(q^{r-d};q^d)_kq^{dk}}{(q^d;q^d)_k^d}
\equiv\frac{(1-q^r)^{r}(1-q^{d-r})(q^d;q^d)_{n-1-(n+r)/d}}{-(-1)^{n-1-(n+r)/d}(q^d;q^d)_{(n+r)/d}^{d-1}}q^{A(d,n,r)},
\end{align}
where $A(d,n,r)=[d(d+n)(n+r)+dn(r-1)-(n+r)^2]/(2d)-r(r+1)/2$. Moreover, the congruence \eqref{4.1} also holds for  $d=2, 3$   and  $r=1$.
\end{theorem}
\begin{theorem}\label{thm:3}
Let $d,  r$ be positive integers with $d>r$ and  $\gcd(d,r)=1$. Let $n>1$ be a positive integer   such that  $n\equiv-r\pmod d$. Then, modulo $\Phi_n(q)^2$,
\begin{align}\label{4.2}
\sum_{k=0}^{n-1}\frac{(q^{d+r};q^d)_k^{d-r-1}(q^r;q^d)_k^{r+1}q^{dk}}{(q^d;q^d)_k^d}
\equiv\frac{(1-q^r)^{r+1}(q^d;q^d)_{n-1-(n+r)/d}}{(-1)^{n-1-(n+r)/d}(q^d;q^d)_{(n+r)/d}^{d-1}}q^{A(d,n,r)-r},
\end{align}
where $A(d,n,r)$ is stated in Theorem \ref{thm:4}.
\end{theorem}

 Clearly,  taking $r=1$ in Theorems \ref{thm:4} and \ref{thm:3}, we can easily attain \eqref{odd}-\eqref{eq:conjecture 5.3}   on the basis of  the parity of $d$.

In fact, the proof of Theorem \ref{thm:4} depends on the relationship between \eqref{4.1} and  \eqref{4.2}. So, we   first prove Theorem \ref{thm:3} via Karlsson-Minton type summation \eqref{eq:KM} again.
\begin{proof}[Sketch of proof of Theorem \ref{thm:3}]
Totally classified by the parities of $d$ and $r$, we    discuss  the following four $q$-congruences \eqref{2(a)}--\eqref{eq:parametric generalization conjecture 5.3 when d=r+2}, which are all symmetric with respect to $a$ and $a^{-1}$.\\
{$(\romannumeral1)$} For $d+r\equiv 1\pmod{2}$,     modulo $(1-aq^n)(a-q^n)$,
\begin{align}\label{2(a)}
&\sum_{k=0}^{n-1}\frac{(a^{d-1}q^{d+r},a^{d-3}q^{d+r},\cdots,a^{r+2}q^{d+r},a^rq^r,a^{r-2}q^r,\cdots,a^{-r}q^r,\cdots,a^{1-d}q^{d+r};q^d)_k}
{(a^{d-2}q^d,a^{d-4}q^d,\cdots,a^{4-d}q^d,a^{2-d}q^d;q^d)_k(q^d;q^d)_kq^{-dk}}\notag\\
&\equiv \frac{(1-a^{r}q^r)(1-a^{r-2}q^r)\cdots(1-a^{-r}q^r)(q^d;q^d)_{n-1-(n+r)/d}}
{(-1)^{n-1-(n+r)/d}(a^{d-2}q^d,a^{d-4}q^d,\cdots,a^{4-d}q^d,a^{2-d}q^d;q^d)_{(n+r)/d}}q^{A(d,n,r)-r}
\end{align}
with $d-r\geq 3$, and
\begin{align}\label{2(b)}
\sum_{k=0}^{n-1}\frac{(a^{d-1}q^{d-1},a^{d-3}q^{d-1},\cdots,a^{3-d}q^{d-1},a^{1-d}q^{d-1};q^d)_kq^{dk}}
{(q^d,a^{d-2}q^d,a^{d-4}q^d,\cdots,a^{4-d}q^d,a^{2-d}q^d;q^d)_k}
\equiv B_q(d,n,d-1)
\end{align}
with $d-r=1$. Here $B_q(d,n,r)$ denotes the right-hand side of \eqref{2(a)}.\\
{$(\romannumeral2)$} For $d\equiv r\equiv 1\pmod{2}$,  modulo $(1-aq^n)(a-q^n)$,
\begin{align}\label{eq:parametric generalization conjecture 5.3}
&\sum_{k=0}^{n-1}\frac{(a^{d-1}q^{d+r},a^{d-3}q^{d+r},\cdots,a^{r+3}q^{d+r},a^{r+1}q^r,\cdots,a^2q^r,q^{d+r},a^{-2}q^r,\cdots,a^{1-d}q^{d+r};q^d)_k}
{(a^{d-2}q^d,a^{d-4}q^d,\cdots,a^{4-d}q^d,a^{2-d}q^d;q^d)_k(q^d;q^d)_kq^{-dk}}\notag\\
&\equiv C_q(d,n,r)\qquad\qquad\qquad\qquad\qquad\qquad\qquad\qquad\quad\text{with $d-r\geq 4$},\\
&\sum_{k=0}^{n-1}\frac{(a^{d-1}q^{d-2},a^{d-3}q^{d-2},\cdots,a^2q^{d-2},q^{2d-2},a^{-2}q^{d-2},\cdots,a^{1-d}q^{d-2};q^d)_kq^{dk}}
{(a^{d-2}q^d,a^{d-4}q^d,\cdots,a^{4-d}q^d,a^{2-d}q^d;q^d)_k(q^d;q^d)_k}\notag\\
&\equiv C_q(d,n,d-2)\qquad\qquad\qquad\qquad\qquad\quad\qquad\qquad\text{with $d-r=2$},\label{eq:parametric generalization conjecture 5.3 when d=r+2}
\end{align}
 where $$
C_q(d,n,r)=
\frac{(q^d;q^d)_{n-1-(n+r)/d} q^{A(d,n,r)-r}}
{(a^{d-2}q^d,a^{d-4}q^d,\cdots,a^{2-d}q^d;q^d)_{(n+r)/d}}\prod_{j=1}^{(r+1)/2}(1-a^{2j}q^r)(1-a^{-2j}q^r).
$$
Actullay, the proof of per condition needs the help of   Karlsson-Minton type summation \eqref{eq:KM}. Taking the congruence \eqref{2(a)} for example,  it follows by letting
$q\rightarrow{q^d}$, $m=d-1$, $N=n-1-(n+r)/d$, $b_j=q^{d-(d-2j)n}$\text{ }$(1\leq j\leq d-1)$, and
\begin{align*}
n_j=
\begin{cases} (n+r)/d-1 &(d-r-1)/2\leq j \leq (d+r-1)/2 \\[4pt] (n+r)/d  &1 \leq j < (d-r-1)/2 \:\: \text{ and } \:\: (d+r-1)/2 < j \leq d-1
\end{cases}
\end{align*}
in \eqref{eq:KM}. Similarly, the other three congruences \eqref{2(b)}-\eqref{eq:parametric generalization conjecture 5.3 when d=r+2} can also be  proved and we will not present the details here.
Hence, letting $a\rightarrow 1$ in \eqref{2(a)}--\eqref{eq:parametric generalization conjecture 5.3 when d=r+2} and making an integration, we are led to Theorem \ref{thm:3}.
\end{proof}
\begin{proof}[Proof of Theorem \ref{thm:4}]
In view of
\begin{align*}
(q^{d+r},q^{r-d};q^d)_k=-q^r\frac{[d-r]}{[r]}\left(1+\frac{1-q^d}{q^d-q^{dk+r}}\right)(q^r;q^d)_k^2,
\end{align*}
we investigate the   relation  between \eqref{4.1} and  \eqref{4.2}: for    integers $d> r\geq1$,
\begin{align*}
&\sum_{k=0}^{n-1}\frac{(q^{d+r};q^d)_k^{d-r}(q^r;q^d)_k^{r-1}(q^{r-d};q^d)_kq^{dk}}{(q^d;q^d)_k^d}\notag\\
&=\frac{[d]}{[r]}\sum_{k=0}^{n-1}\frac{(q^{d+r};q^d)_k^{d-r-1}(q^r;q^d)_k^r(q^{r-d};q^d)_kq^{dk}}{(q^d;q^d)_k^d}
-\frac{[d-r]}{q^{-r}[r]}\sum_{k=0}^{n-1}\frac{(q^{d+r};q^d)_k^{d-r-1}(q^r;q^d)_k^{r+1}q^{dk}}{(q^d;q^d)_k^d}.
\end{align*}
Hence, from \eqref{eq:lemma1}, \eqref{22}, \eqref{33}, Theorem \ref{thm:3} and the above result, we acquire Theorem \ref{thm:4}.
\end{proof}

Particularly, letting $n$ be  a prime  $p\geq 5$ and $q\rightarrow 1$ in Theorems \ref{thm:4} and \ref{thm:3},
we gain the following conclusion.

\begin{corollary}\label{corol:3}
{$(\romannumeral1)$}
Let $d$, $r$ be positive integers with  $d\geq 3+r$ and  $\gcd(d,r)=1$. Then, for any   prime $p\geq  2d-r$ with $p\equiv-r \pmod d$,
\begin{align*}
\sum_{k=0}^{p-1}\frac{(\frac{d+r}{d})_k^{d-r}(\frac{r}{d})_k^{r-1}(\frac{r-d}{d})_k}{k!^d}
\equiv \tfrac{d-r}{d}\big(\tfrac{r}{d}\big)^{r}\Gamma_p \big(-\tfrac{r}{d} \big)^d
\pmod{p^2}.
\end{align*}
{$(\romannumeral2)$}
Let $d$, $r$ be positive integers with  $d>r$ and  $\gcd(d,r)=1$. Then, for any prime $p\geq 5$ with $p\equiv-r \pmod d$,
\begin{align*} \sum_{k=0}^{p-1}\frac{(\frac{d+r}{d})_k^{d-r-1}(\frac{r}{d})_k^{r+1}}{k!^d}
\equiv -\big(\tfrac{r}{d}\big)^{r+1}\Gamma_p \big(-\tfrac{r}{d} \big)^d
\pmod{p^2}.
\end{align*}
\end{corollary}
In the derivation of Corollary \ref{corol:3}, we have used the following formula: for any positive integers  $d, r$ with $d>r$ and prime $p\geq5$  with $p\equiv-r\pmod d$,
\begin{align*}
\frac{\big(p-1-\frac{p+r}{d}\big)!}{(\frac{p+r}{d})!^{d-1}}
\equiv-(-1)^{\frac{p+r}{d}}\Gamma_p \big(-\tfrac{r}{d} \big)^d
\pmod{p^2},
\end{align*}
which   can be proved like the formulas  \cite[(4.4) and (4.8)]{Guoqdrrama2022} by means of some basic properties of the $p$-adic Gamma function.

\end{document}